\documentclass[12pt]{article}

\usepackage{amsmath}
\usepackage{amssymb}

 \newcommand{\beq}{\begin{equation}}
\newcommand{\eeq}{\end{equation}}

\newtheorem{theorem}{Theorem}[section]
\newtheorem{proposition}{Proposition}[section]
\newtheorem{lemma}[theorem]{Lemma}
\newtheorem{corollary}[theorem]{Corollary}

\newenvironment{proof}{\medbreak\noindent{\it Proof.}\rm}{\hfill$\square$\rm}

\newcommand{\PSH}{{\operatorname{PSH}}}
\newcommand{\MC}{{\operatorname{MC_0}}}
\newcommand{\CNV}{{\operatorname{CNV}}}
\newcommand{\D}{{\mathbb D}}
\newcommand{\Z}{{\mathbb Z}}
\newcommand{\Rp}{{\mathbb R}^p}
\newcommand{\Rn}{{\mathbb R}^n}
\newcommand{\Rnm}{{\mathbb R}_-^n}
\newcommand{\Rnp}{{\mathbb R}_+^n}
\newcommand{\C}{{\mathbb  C}}

\newcommand{\Cn}{{\mathbb  C\sp n}}
\newcommand{\F}{{\mathcal F}}
\newcommand{\cI}{{\mathcal  I}}
\newcommand{\cJ}{{\mathcal  J}}
\newcommand{\cO}{{\mathcal  O}}
\newcommand{\mi}{{\mathfrak{m}}}

\title{Multi-circled singularities, Lelong numbers, and integrability index}
\author{Alexander Rashkovskii}
\date{}

\begin{document}

\maketitle

\begin{abstract} By comparing Green functions of multi-circled plurisubharmonic singularities $u$ in $\Cn$ to their indicators, we prove formulas for higher Lelong numbers $L_k(u)$ and integrability index $\lambda_u$ (the latter one being due to Kiselman) and extend Howald's result on multiplier ideals for monomial ideals to multi-circled singularities. This also leads to an elementary proof of the relations $\lambda_u\le k^{-1}L_k(u)^{1/k}$, $1\le k\le l$, for the multi-circled singularities, where $l$ is the codimension of the set $u^{-1}(-\infty)$. For $k=1$ and arbitrary plurisubharmonic function $u$ the inequality is due to Skoda, and for $k=n$ and any plurisubharmonic $u$ with isolated singularity the relation is due to Demailly. We also describe all multi-circled functions for which the inequalities are equalities.

 We also prove these inequalities, by a reduction to Demailly's result, in the general case of (not necessarily multi-circled) plurisubharmonic functions. In addition, we get a description of all plurisubharmonic singularities $u$ whose integrability index is given by the lower bound in Skoda's inequality, i.e., $\lambda_u= n^{-1}L_1(u)$.

\end{abstract}

\section{Introduction and results}

Recall that an upper semicontinuous function $u$ on a complex manifold $M$
is called {\it plurisubharmonic} ({\it psh}) if for every holomorphic mapping $\xi$
from the unit disk into $M$, the function $u\circ\xi$ is subharmonic. A basic example is
$u=c\log|f|$ with $c>0$ and a function $f$ holomorphic on $M$ -- or, more generally, a holomorphic mapping $f:M\to{\C}^p$ ($|\cdot|$ means the Euclidean norm.)

We restrict ourselves to local considerations, so in the sequel we deal with functions defined near
$0\in{\C}^n$. Let ${\mathcal O}_0$ denote the ring of germs of analytic functions $f$
at $0$, and let $\mi=\{f\in {\cal O}_0 :
f(0)=0\}$ be its maximal ideal. The log-transformation
$f\mapsto\log|f|$ maps ${\mathcal O}_0$ into the collection $\PSH_0$ of germs of
psh functions. We will say that a psh germ $u$ is {\it
singular} at $0$ if $u$ is not bounded in any neighborhood of
$0$. For functions $u=\log|f|$, $f\in {\cal O}_0 $, this means $f\in\mi$.
Asymptotic behavior of arbitrary psh functions can be much more
complicated; nevertheless, all their standard "rough" characteristics can be viewed as extensions of the corresponding notions for analytic functions.

\medskip

A basic characteristic of singularity of $u\in
\PSH_0$ is its {\it Lelong number}
$$ \nu_u=
\liminf_{z\to 0}\frac{u(z)}{\log|z|}.$$
If $f\in\mi$, then $\nu_{\log|f|}=m_f$, the multiplicity (vanishing order) of $f$ at $0$. If ${\cI}=\langle f_1,\ldots,f_p\rangle$ is the ideal in ${\cal O}_0$ generated by $f_1,\ldots,f_p$, then $$m_{\cI}=\min\{m_f:\: f\in {\cI}\}=\nu_{\max_k\log|f_k|}=\nu_{\log|f|}.$$

Let $u\in\PSH_0$. Assuming $(dd^c u)^k$
well-defined for all $k\le l$, denote
\begin{equation}\label{eq:Lk} L_k(u)=\nu((dd^c u)^k,0)=(dd^c u)^k\wedge(dd^c\log|z|)^{n-k}(0),\end{equation}
the {\it higher order Lelong numbers} of $u$, that is, the Lelong numbers of the currents $(dd^c u)^k$, at $0$. Here  $d=\partial + \bar\partial$, $d^c= (
\partial -\bar\partial)/2\pi i$, so $(dd^c\log|z|)^{n}$ charges $0$ with the unit mass.

In particular, $L_1(u)=\nu_u$ for any $u\in\PSH_0$, and if $(dd^c u)^n$ is well defined near $0$, then $L_n$ is
the Monge-Amp\`ere mass of $u$ at $0$: $L_n(u)=(dd^c u)^n(0)$. If $u=\log|f|$ for a holomorphic mapping $f$ whose zero set has codimension $l$ at $0$, then $L_l(u)$ equals the multiplicity of the mapping $f$ at $0$.
For an $\mi$-primary (zero-dimensional) ideal $\cI$ generated by $f_1,\ldots,f_p$, its Samuel multiplicity equals $L_n(\log|f|)$, and the Lelong numbers $L_k(\log|f|)$ are the mixed multiplicities of $n$-tuples of ideals consisting of $k$ copies of $\cI$ and $n-k$ copies of $\mi$.

\medskip
One more characteristic, introduced in various contexts by several authors (first, probably, in \cite{Sk}) and attracted recently considerable attention (e.g., \cite{ACKPZ}, \cite{Be}, \cite{D8}, \cite{DK}, \cite{FaJ2}), \cite{Gu}, \cite{Kis3}), is the {\it integrability index (at $0$)}
\begin{equation}\label{ii} \lambda_u=
\inf\{\lambda>0:e^{-u/\lambda}\in L^2_{loc}(0)\}.\end{equation} If
$f=(f_1,\ldots,f_p)\in\mi^p$, the value
$\lambda_{\log|f|}$ is known as the {\it Arnold multiplicity} of the
ideal ${\cI}$ generated by $f_j$, and its inverse
\begin{equation}\label{lc}
lc({\mathcal I})=\lambda_{\log|f|}^{-1}\end{equation} is the {\it
log canonical threshold} of ${\cI}$. The value of $lc(\cI)$ can be "computed" by means of log resolution of $\cI$ (see, e.g., \cite{K}).

\medskip
A classical result due to Skoda \cite{Sk} states that:
\begin{equation}\label{skoda}
n^{-1}\nu_u\le \lambda_u\le \nu_u,
\end{equation}
the extremal situation being realized, for example, for
$u=\log|z|$ (for the first inequality) and $u=\log|z_1|$ (for the second one).

Recently, in  \cite{dFEM} and \cite{M}, the log canonical threshold of a zero dimensional ideal ${\mathcal I}$ was related to its Samuel multiplicity $e({\mathcal I})$:
\begin{equation}\label{dFEM} lc({\cI})\ge\frac{n}{e(\cI)^{1/n}},\end{equation}
with an equality if and only if the integral closure of $\cI$ is a power of the maximal ideal.
It was used in \cite{D8} for a corresponding bound for psh functions $u$ with isolated singularity at $0$, and extended then in \cite{Ze} to all $u$ with $(dd^cu)^n$ well defined,
\begin{equation}\label{dem} \lambda_u\le n^{-1}\,{L_n(u)^{1/n}}.\end{equation}
A direct proof of Demailly's inequality (\ref{dem}) without using (\ref{dFEM}) was obtained then in \cite{ACKPZ}. Notice that (\ref{dem}) implies (\ref{dFEM}) by setting $u=\log|f|$ with $f=(f_1,\ldots,f_p)$ generators of $\cI$. To the best of our knowledge, the question of equality in (\ref{dem}) is open.
Observe also that none of the bounds (\ref{skoda}) and (\ref{dem}) implies the other one.

\medskip

While computation of the standard Lelong number is quite easy, this is not the case for the higher Lelong numbers and the integrability index. It is thus desirable to simplify this, say, by replacing the function with another one whose asymptotic behavior is easier to handle. For example, given a function $u\in\PSH^-(D)$ (which means that it is negative and plurisubharmonic in a domain $D$) with isolated singularity at $0$, its residual Monge-Amp\`ere mass at $0$ can be shown to coincide with that of $g_u$, the regularized upper envelope of all functions $v\in\PSH^-(D)$ such that $v\le u+O(1)$ \cite{R7}. The function $g_u$ (the {\it greenification} of $u$) satisfies $(dd^cg_u)^n=0$ outside $0$, so its total mass in $D$ coincides with that at $0$.

This function keeps also the higher Lelong numbers and the integrating factor, see section~2. Observe however that we do not know if the relation $L_k(u)=L_k(g_u)$ holds true for all $u$ such that $(dd^cu)^k$ is well defined.

We consider then {\it multi-circled} psh singularities
$u(z)=u(|z_1|,\ldots, |z_n|)$; their collection will be denoted by $\MC$. They can be considered as a plurisubharmonic counterpart for the notion of {\it monomial ideals} (those that can be generated by monomials). On the other hand, as was shown in \cite{R7}, the function $\log|F|$ has, up to a bounded term, multi-circled singularity for generic holomorphic mappings $F$. In Section~3.1 we show that the greenification $g_u$ of $u\in\MC$ in the unit polydisk $\D^n$  coincides with the {\it (local) indicator} $\Psi_u$ introduced in \cite{LeR}. As a consequence, $L_k(u)=L_k(\Psi_u)$ (whenever defined) and $\lambda_u=\lambda_{\Psi_u}$. This gives relatively easy computations of these characteristics for the indicators: for higher Lelong numbers, by means of mixed Minkowski's (co)volumes, and for the integrability index -- by means of the directional Kiselman-Lelong numbers, see Sections 3.2 and 3.3. The latter one repeats a result of Kiselman \cite{Kis3}) and is used in Section 3.4 for a description of the multiplier ideals for multi-circled singularities in the spirit of Howald's result for monomial ideals \cite{Ho}; a similar extension was made, as we recently learned, in \cite{Gu}.

Finally, we obtain relations between the Lelong numbers $L_k(u)$ and the integrability index $\lambda_u$, that fill the gap between Skoda's (\ref{skoda}) and Demailly's (\ref{dem}) inequalities,
\beq\label{eq:rel} \lambda_u\le k^{-1}L_k(u)^{1/k},\quad 1\le k\le l,\eeq
where $l$ is the codimension of an analytic set $A$ such that $u^{-1}(-\infty)\subset A$. In Section 3.5 we do that for $u\in\MC$ as a consequence of the corresponding relations for the indicator $\Psi_u$; in addition, we describe all $u\in\MC$ with equality in (\ref{eq:rel}) for some $k\le n$: these are those whose indicators have the form
 $\Psi_u(z)=B\max_{j\in J}\log|z_{j}|$ with $B\ge 0$ and $J=(j_1,\ldots,j_k)$.

 In Section~4, relations (\ref{eq:rel}) are proved for the general (not necessarily multi-circled) case, by reduction to Demailly's inequality (\ref{dem}). As a consequence, we get bounds for the log canonical threshold of an ideal in terms of mixed Samuel multiplicities. In addition, we describe all plurisubharmonic singularities $u$ whose integrability index is given by the lower bound in Skoda's inequalities (\ref{skoda}), i.e., $\lambda_u= n^{-1}\nu_u$.

\section{Green and Green-like functions}

Let $D\subset\Cn$ be a bounded hyperconvex neighborhood of $0$, and let $u\in\PSH^-(D)$, $u(0)=-\infty$.
Consider the class
$\F_{u,D}$ of negative psh functions $v$ in $D$ such that $v(z)\le u(z)+O(1)$ near $0$,
then the regularization of its upper envelope,
 \beq\label{eq:cgf} g_u(z)=g_{u,D}(z) = \limsup_{y\to
z}\sup\{v(y):\: v\in \F_{u,D}\},\eeq
is a psh function in $D$, maximal on $D\setminus \{g_u=-\infty\}$. It is called the {\it
complete greenification} of $u$ at $0$ in $D$ \cite{R7}.

If $u$ is locally bounded and maximal on a punctured neighborhood of $0$, then $g_{u}=G_{u}$, the Green--Zahariuta function for $u$. In this case we have
\beq\label{eq:ugu} u= g_u+O(1)\quad {\rm near\ }u^{-1}(-\infty).\eeq
Relation (\ref{eq:ugu}) holds true if $u=\log|f|+O(1)$ for a holomorphic mapping $f: {\mathbb C}^n_0\to{\mathbb C}^p_0$, even without the maximality and isolated singularity assumptions; in this situation, $g_u=G_\cI$, the {\it pluricomplex Green function for the ideal} $\cI$ generated by the components of the mapping $f$, see \cite{RSig2}.

In the general case of non-maximal singularity, even isolated, relation (\ref{eq:ugu}) is no longer true; nevertheless, this does not affect the Lelong numbers $L_k$ (\ref{eq:Lk}).

\begin{proposition}\label{theo:Lngu} If $u$ has isolated singularity at $0$, then
$$ L_k(u)=L_k(g_u),\quad 1\le k\le n.$$
\end{proposition}

\begin{proof} By the Choquet lemma, there exists a sequence $u_j\in\F_{u,D}$ increasing a.e. to $g_u$; obviously they can be chosen to have isolated singularity at $0$. Therefore, the currents $(dd^cu_j)^k$ converge to $(dd^cg_u)^k$, see \cite[Thm. 5]{X} (the statement of that theorem is on the convergence of $(dd^cu_j)^n$ only, while the proof uses induction in the degree $k$).

By Demailly's Semicontinuity theorem for Lelong numbers \cite{D1}, this implies $$\limsup_{j\to\infty} L_k(u_j)\le L_k(g_u).$$ The relations $u_j\le u+O(1)\le g_u$ give, by Demailly's Comparison theorem \cite{D1}, $L_k(u)\le \limsup L_k(u_j)$ and $L_k(g_u)\le L_k(u)$.
\end{proof}

\medskip

{\it Remarks.} 1. For $k=n$, this was proved in  \cite[Thm. 1]{R4}.

\medskip
2. We do not know if $L_k(u)=L_k(g_u)$ for any $u$ with well-defined $(dd^cu)^k$.

\medskip

The greenification keeps also the value of of the integrability index, without any isolated singularity assumption.

\begin{proposition}\label{theo:iind} For any $u\in\PSH_0$, $\lambda_u=\lambda_{g_u}$.
\end{proposition}

\begin{proof} This follows from the Semicontinuity theorem for the integrability index \cite{DK}, applied to a sequence $u_j\in\F_{u,D}$ increasing a.e. to $g_u$.
\end{proof}

\section{Multi-circled singularities}

In this section we work with {\it multi-circled singularities} $u\in\MC$, that is, germs $u\in\PSH_0$ such that $u(z)=u(|z_1|,\ldots, |z_n|)$. The function $$\widehat u(t)=u(e^{t_1},\ldots, e^{t_n})$$ is then convex on the set $\{t\in\Rn: \: t_j\le -N,\ 1\le j\le n\}$ for some $N>0$ and increasing in each of the variables $t_j$; we will call it the {\it convex image} of $u$.

Denote ${\bf 1} =(1,\ldots,1)$. Since the function $\widehat u(t-2N{\bf 1})-\widehat u(t)$ is decreasing in each $t_j$, one has
$\widehat u(t)+ C\le \widehat u(t-2N{\bf 1})\le \widehat u(t)$ for all $t$, so we can always assume $\widehat u$ defined and negative on $\Rnm=\{t\in\Rn: \: t_j<0,\ 1\le j\le n\}$ and, accordingly, $u$ defined (and negative) on the unit polydisk $\D^n$.

The collection of all negative convex functions on $\Rnm$, increasing in each variable, will be denoted by $\CNV_-$.

\subsection{Indicators and greenifications}\label{ssec:ind}

We fix the domain $D$ to be the unit polydisk $\D^n$. Given $u\in\PSH^-(\D^n)$, we would like to compare the corresponding greenification
$g_u$ with the {\it indicator} of $u$ defined in \cite{LeR} as follows.

 Let $\nu_u(a)$ denote the Lelong-Kiselman number of $u$ in the direction $a\in\Rnp=\{a\in\Rn:\: a_j>0,\ 1\le j\le n\}$, that is,
$$\nu_u(a)=\liminf_{x\to0}\frac{u(x)}{\phi_a(x)},$$
where
\beq \label{eq:phia} \phi_a(x)=\max_k a_k^{-1}\log|x_k|.\eeq
Then the function
\beq\label{eq:defind} \Psi_u(x)=\nu_u(-\log|x_1|,\ldots,-\log|x_n|)\eeq
extends from $\D_*^n=\D^n\setminus \{x_1\cdot\ldots\cdot x_n\neq 0\}$ to a function plurisubharmonic in $\D^n$, the {\it indicator of} $u$. Its convex image $\widehat\Psi_u(t)=\Psi_u(e^{t_1},\ldots, e^{t_n})$ is a convex function on $\Rnm$, satisfying $\widehat\Psi_u(c\,t)=c\,\widehat\Psi_u(t)$ for any $c>0$. In particular, it is the restriction to $\Rnm$ of the characteristic function of a closed convex set $\Gamma_u\subset{\overline\Rnp}$,
\beq\label{eq:charf} \widehat\Psi_u(t)=\sup_{a\in\Gamma_u}\langle a,t\rangle,\eeq
with the property $\Gamma_u +\Rnp\subset\Gamma_u$, so
 \beq\label{eq:charf1} \Gamma_u=\{a\in{\overline\Rnp}:\: \langle a,t\rangle\le\widehat\Psi_u(t)\quad \forall t\in\Rnm.\}\eeq
 We will call $\Gamma_u$ the {\it indicator diagram} of $u$. When $u=\log|f|$ for $f\in\cO_0$, $\Gamma_u$ is the Newton polyhedron of $f$ at $0$ in the sense of Kushnirenko.

Every function $u\in\PSH^-(\D)$ satisfies the relation $u\le\Psi_u$ \cite{LeR}; moreover, it dominates every function $v$ participating in the definition of the $g_u$, so
\beq\label{eq:indbound} g_u\le\Psi_u.\eeq

In the case of multi-circled singularities, the indicators can be defined the other way round. Namely, convexity of $\widehat u$ implies that the ratio
\beq\label{eq:ratio} h_u(c)=\frac{\widehat u(ct)-\widehat u(t)}{c-1}\eeq
is increasing in $c$ for any $t\in\Rnm$, so
we introduce first the function $\widehat\Psi_u$ as the limit of $h_u(c)$ as $c\to\infty$ or, equivalently,
\beq\label{eq:indmc} \widehat\Psi_u(t)=\lim_{c\to\infty} c^{-1}\,\widehat u(ct), \quad t\in\Rnm,\eeq
then we set $\Psi_u(x)=\widehat\Psi_u(\log|x_1|,\ldots,\log|x_n|)$ on $\D_*^n$ and extend it to the whole $\D^n$.

Note also that the greenification $g_u$ of $u\in\MC$ can be also defined equivalently by considering first the envelope
\beq\label{eq:convgr} g_{\hat u}=\sup \{\widehat v\in \widehat\F_u\}, \eeq
where $$\widehat\F_{ u}=\{\widehat v\in \CNV_-:\: \widehat v(t)\le\widehat u(t)+O(1){\rm\ as\ }|t|\to\infty\},$$
and then writing it down as the convex image of $g_u$, i.e., $g_{\hat u}=\widehat g_u$.

\begin{theorem}\label{prop:Grind} $g_u=\Psi_u$ for any $u\in\MC$.
\end{theorem}

\begin{proof} In view of (\ref{eq:indbound}), we need to prove only the inequality $g_u\ge\Psi_u$ or, actually, $g_{\hat u}\ge\widehat\Psi_u$.

For any $c>0$, denote $T_c\widehat u(t)=c^{-1}\widehat u(ct)$.
The function $T_c{\hat u}$ belongs to the class $\CNV_-$ and has the same indicator $\widehat \Psi_u$, however it need not lie in the class $\widehat \F_{ u}$.

 Let $s>0$. Since both $T_s\widehat u$ and $\widehat\Psi_u$ are continuous functions on the simplex $\Sigma=\{\theta\in\overline\Rnm:\: \sum\theta_j=-1\}$, the mentioned monotonicity of the ratios (\ref{eq:ratio}) implies uniform convergence of $\widehat T_su$ to $\widehat\Psi_u$ on $\Sigma$ as $s\to\infty$. Therefore, for any $\epsilon>0$ one can find $s_0>0$ such that $|T_su-\widehat \Psi_u|\le \epsilon/2$ on $\Sigma$ for all $s>s_1$. Similarly, we get the relation  $|T_s(T_c\widehat u)-\widehat \Psi_u|\le \epsilon/2$ on $\Sigma$ for all $s>s_c$ and so,
 $ T_c\widehat u(t)\le\widehat u(t)+\epsilon|t|$ for all $t\in\Rnm$ with $|t|\ge \max\{s_1, s_c\}$, which implies
 $$ T_c\widehat u(t)+\epsilon\sum_jt_j\le\widehat u(t),\quad |t|\ge \max\{s_1, s_c\}.$$
 Therefore, the function $ T_c\widehat u(t)+\epsilon\sum_jt_j$ belongs to the class $\widehat \F_u$, which gives us 
 $ T_c\widehat u(t)+\epsilon\sum_jt_j\le g_{\hat u}$. Since $\epsilon$ is arbitrary, relation (\ref{eq:indmc}) completes the proof.
\end{proof}

\subsection{Monge-Amp\`ere masses}

Given $k\le n$, the Monge-Amp\`ere operator $(dd^cu)^k$ is well defined on plurisubharmonic functions  $u$ whose unbounded locus $UL(u)$ has codimension at least $ k$ \cite{D1}. For $u\in\MC$ this means
\beq \label{eq:codiml} UL(u)=\{z:\:z_{i_j}=0, \ 1\le j\le l\},\quad l\ge k. \eeq
For such a function, the $k$-th Lelong number $L_k$ is defined by (\ref{eq:Lk}).

\begin{theorem}\label{thm:MAn}
Let $u\in\MC$ satisfy (\ref{eq:codiml}), then
\begin{equation}\label{eq: MAm}L_k(u)= L_k(\Psi_u),\quad 1\le k\le l.\end{equation}
\end{theorem}

\begin{proof}
Proposition \ref{theo:Lngu} and Theorem~\ref{prop:Grind} imply (\ref{eq: MAm}) for $u\in\MC$ with $UL(u)=\{0\}$.

If $\dim u^{-1}(-\infty)>0$, consider restrictions $u|_S$ to $k$-dimensional linear spaces $S$ such that $u|_S$ has isolated singularity. By Siu's theorem (see, e.g., \cite[Thm. 5.13]{D1}), for almost all $S\in G(k,n)$ one has $L_k(u|_S)=L_k(u)$ and $\Psi_{u|_S}=\Psi_u|_S$. Then we refer to the already proved case of isolated singularity.
\end{proof}

\medskip

Note that these relations are nontrivial even in the class of multi-circled singularities, because there exist $u\in\MC$, locally bounded outside $0$ and such that $\limsup_{z\to 0} u(z)/\Psi_u(z)>1$. For example, we can take $$u(z_1,z_2)=\max\{-|\log|z_1||^{1/2},\log|z_2|\} + \max\{\log|z_1|,\log|z_2|\}$$ and observe that $\Psi_u=\max\{\log|z_1|,\log|z_2|\}$, while $\limsup_{z\to 0} u(z)/\Psi_u(z)=2$.
\medskip

One of the benefits of formula (\ref{eq: MAm}) is a possibility of computing the residual Monge-Amp\`ere mass $L_k(u)$ as the (mixed) covolume.

Given a convex set $A\subset\Rnp$, we define its {\it covolume} as $${\rm Covol}\,(A)=
{\rm Vol}\,(\Rnp\setminus A),$$
and consider a form ${\rm Covol}\,(A_1,\ldots,A_n)$ on
$n$-tuples of complete convex subsets $A_1,\ldots,A_n$ of $\Rnp$,
multilinear with respect to Minkowski's addition and such that for
any $A$ with bounded complement in $\Rnp$ we have ${\rm
Covol}\,(A,\ldots,A)={\rm Covol}\,(A)$.

For any $k\le n$, denote by ${\rm Covol}_k(A)$ the {\it (mixed) covolume} of the $n$-tuple consisting of $k$ copies of the set $A$ and $n-k$ copies of the set $$\Delta=\{a\in\Rnp:\: \sum_ja_j\ge 1\}.$$


As was shown in \cite{R}, for any indicator $\Psi_u$ with isolated singularity at $0$,
\beq\label{eq:LnPsi}L_n(\Psi_u)=n!\,{\rm Covol}\,(\Gamma_u),\eeq where $\Gamma_u$ is the indicator diagram of $u$ (\ref{eq:charf}).
 Since $\Delta=\Gamma_{\log|z|}$, the polarization formula for multilinear forms $\mu(s_1,\ldots,s_n)$,
\beq\label{eq:polar} \mu(s_1,\ldots,s_n)=\frac{(-1)^n}{n!}\sum_{j=1}^n\sum_{i_i<\ldots<i_j} \mu\left(\sum_k s_{i_k},\ldots,\sum_k s_{i_k}\right), \eeq
applied to the Monge-Amp\`ere operators and covolumes, extends (\ref{eq:LnPsi}) to the mixed setting, so Theorem~\ref{thm:MAn} implies

\begin{corollary}\label{cor:covol} If $u\in\MC$ satisfies (\ref{eq:codiml}), then
$$ L_k(u)=n!\, {\rm Covol}_k (\Gamma_u),\ \quad 1\le k\le l.$$
\end{corollary}

\subsection{Integrability index and directional Lelong numbers}

Here we use Proposition~\ref{theo:iind} and Theorem~\ref{prop:Grind} for re-proving Kiselman's result \cite{Kis3} on computation of the integrability index $\lambda_u$ (\ref{ii}) of a multi-circled function $u$ by means of its Lelong-Kiselman numbers.

In what follows, we will repeatedly use the following elementary

\begin{lemma}\label{lem:conv} Let $$F(r)=\int_Ae^{-rg(t)}\,dV_p(t),\quad r>0,$$ where
$g$ is a bounded convex function on a bounded convex open set $A\subset\Rp$. Then
\beq\label{eq:genconv}\int_0^\infty F(r)\,dr<\infty\eeq
if and only if $\inf_A g>0$.
\end{lemma}

\begin{proof} Evidently, (\ref{eq:genconv}) implies $g\ge 0$ on $A$. Let $g(t_0)=0$ for a point $t_0\in\overline A$. Since $g$ is convex, the overgraph of $g$ contains the intersection of a neighborhood $U$ of $t_0$ with a cone with vertex at $t_0$. By a coordinate change, this means $0\le g(t)\le \gamma|t|$ for $t\in U\cap A$. Therefore, there exists $C>0$ such that $F(r)\ge cr^{-1}$ for all $r>0$, which contradicts (\ref{eq:genconv}).
The reverse implication is evident.
\end{proof}

\medskip

It is easy to compute the integrability index
 for an {\it indicator}, i.e., a multi-circled function $\Psi\in\PSH^-(\D^n)$ whose convex image $\widehat\Psi$ satisfies $\widehat\Psi(ct)=c\,\widehat\Psi(t)$ for every $c>0$.

\begin{lemma}\label{prop:iiind} For any indicator $\Psi$, its integrability index $\lambda_\Psi=\sup_t\widehat\Psi(t)/\sum t_j$, and $\Psi/\lambda_\Psi\not\in L^2_{loc}(0)$.
\end{lemma}

\begin{proof}
Since $\Psi$ depends only on
absolute values of the variables, it suffices to check the integral
over the unit polydisk $\D^n$. By transition to the function
$\widehat\Psi(t)$, we get for any $\lambda>0$,
$$\frac1{(2\pi)^n}\int_{\D^n}e^{-2\Psi/\lambda} \,dV_{2n}=
\int_{\mathbb R_+^n}e^{2\left(\sum t_j-\widehat\Psi(t)/\lambda\right)}
\,dV_n = \int_{0}^\infty \int_{\Lambda}
e^{-2r\left(1+\widehat\Psi(t)/\lambda\right)}\,dV_{n-1}(t)\,dr,$$ where
$\Lambda=\{t\in{\mathbb R}_-^n:\: \sum_kt_k=-1\}$. Therefore, by Lemma~\ref{lem:conv}, the
integral converges if and only if
$\lambda>\sup\{-\widehat\Psi(t):\:t\in\Lambda\}$.
\end{proof}

\begin{theorem}\label{prop:lambdau} For any $u\in\MC$,
\begin{equation}\label{eq:lambdau}\lambda_{u}=\lambda_{\Psi_u}=\sup_{a\in\Rnp}\,\frac{\nu_u(a)}{\sum_i a_i}  \end{equation}
and $e^{-u/\lambda_u}\not\in L^2_{loc}(0)$.
\end{theorem}

\begin{proof} By (\ref{eq:defind}), this follows from Proposition~\ref{theo:iind}, Theorem~\ref{prop:Grind}, Lemma~\ref{prop:iiind} and the inequality $u\le \Psi_u+O(1)$.

\end{proof}

\medskip
Note that for $\Psi_u$ with isolated singularity  (\ref{eq:lambdau}) can be rewritten as
\beq\label{eq:ltype}\lambda_u^{-1}=\liminf_{z\to 0} \frac{\log|z_1\ldots z_n|}{\Psi_u(z)},\eeq
representing thus it as the relative type of $\log|z_1\ldots z_n|$ with respect to $\Psi_u$.

\subsection{Multiplier ideals}

Here we use formula (\ref{eq:lambdau}) for a description of multiplier ideals for multi-circled singularities. When the singularity is generated by a monomial ideal, this repeats Howald's results \cite{Ho}. See also \cite{Gu} and \cite{MN}.

Let $\cJ(u)$ be the multiplier ideal of a function $u\in \PSH_0$, i.e., the collection of functions (germs) $f\in\cO_0$ such that $|f|e^{-u}\in L^2_{loc}$ near $0$.

\begin{lemma} If $u\in\MC$, then the ideal $\cJ(u)$ is monomial.
\end{lemma}

\begin{proof} Let $f=\sum c_\alpha z^\alpha\in\cJ(v)$, then for a polydisk $D$ around $0$, we have
$$ \int_D |f|^2e^{-2u}\,dV= \sum_{\alpha,\beta}c_\alpha\overline{c_\beta}\int_D z^\alpha
\bar z^\beta e^{-2u(|z_1|,\ldots,|z_n|)}\,dV(z)= \sum_\alpha |c_\alpha|^2\int_D |z^\alpha|^2e^{-2u}\,dV,$$
which implies $z^\alpha\in\cJ(u)$ for any $\alpha$ such that $c_\alpha\neq 0$.
\end{proof}

\medskip

Let $\Gamma_u\subset\Rnp$ be the indicator diagram (\ref{eq:charf}) of $u$. Denote ${\bf 1}=(1,\ldots,1)\in\Z^n_+$.

\begin{proposition} Let $u\in \MC$. Then $z^\alpha\in\cJ(u)$ iff $\alpha+{\bf 1}\in int\,\Gamma_u$.
\end{proposition}

\begin{proof} We start with showing $\cJ(u)=\cJ(\Psi_u)$. Since $u\le\Psi(u)+O(1)$, we have $\cJ(u)\subset\cJ(\Psi_u)$.

We can assume that $u\in\PSH^-({\D^n})$. For any $W\in \MC\cap\PSH({\D^n})$, the function $w(t)=W(e^{t_1},\ldots,e^{t_n})$ is convex on $\Rnm$.
Take any any $\alpha\in\Z^n_+$, then
\beq\label{eq:alpha}\frac1{(2\pi)^n}\int_{\D^n}|z^\alpha|^2e^{-2W} \,dV_{2n}=
\int_{\mathbb R_+^n}e^{2\left(\langle \alpha+{\bf 1},t\rangle-w(t)\right)}
\,dV_n
= C(\alpha)\int_{\mathbb R_+^n}e^{2\left(\sum s_j-w_\alpha(s)\right)}
\,dV_n,
\eeq
where $c(\alpha)= \prod_j(1+\alpha_j)^{-1}$ and $w_\alpha(s)=w((1+\alpha_1)s_1,\ldots,(1+\alpha_n)s_n)$.

As a consequence, $z^\alpha\in\cJ(W)$ iff the integral of the right hand side in (\ref{eq:alpha}) is finite, and the finiteness of the integral implies $\lambda_{w_\alpha}\le 1$.

Let $u_m$ be a
sequence satisfying $u_m\le u+O(1)$ and increasing a.e. to $g_u=\Psi_u$. Then $\cJ(u_m)\subset\cJ(u)$. On the other hand, the functions $$u_{m,\alpha}(z)=u_m(|z_1|^{1+\alpha_1},\ldots,|z_1|^{1+\alpha_1})$$ increase almost everywhere to $$U_\alpha(z)=\Psi_u(|z_1|^{1+\alpha_1},\ldots,|z_1|^{1+\alpha_1}).$$ By the upper semicontinuity of the
integrability index, $\lambda(u_{m,\alpha})$ decrease to $\lambda(U_\alpha)$. By Proposition~\ref{prop:lambdau},  $e^{-U_\alpha/\lambda_{U,\alpha}}\not\in L^2_{loc}$, so $z^\alpha\in\cJ(\Psi_u)$ implies $\lim\lambda(u_{m,\alpha})< 1$. In turn, this means $z^\alpha\in\cJ(u_m)$, so $\cJ(\Psi_u)\subset\cJ(u_m)\subset\cJ(u)$.

Finally, we need  a description for $\cJ(\Psi_u)$. By the first equality in (\ref{eq:alpha}), $z^\alpha\in\cJ(\Psi_u)$ iff
$$ \int_{\mathbb R_+^n}e^{2\left(\langle \beta,t\rangle-\widehat\Psi_u(t)\right)}
\,dV_n = \int_{0}^\infty \int_{\Lambda_\beta}
e^{-2r(1+\widehat\Psi_u(t))}\,dV_{n-1}(t)\,dr<\infty,
$$
where $\beta=\alpha+{\bf 1}$ and $\Lambda_\beta=\{t\in{\mathbb R}_-^n:\: \langle \beta,t\rangle=-1\}$. By Lemma~\ref{lem:conv}, the integral converges iff $\widehat\Psi_u(t)>-1$ on $\Lambda_\beta$, which is equivalent to $\langle \beta,t\rangle<\widehat\Psi_u(t)$ for all $t\in\Rnm$ and thus to $\beta\not\in int\,\Gamma_u$.
\end{proof}

\subsection{Integrability index and $L_k$}

To compare integrability index $\lambda(u)$ with higher order Lelong numbers $L_k(u)$, $u\in\MC$, we need the following computation of $L_K(\phi_a)$ for the functions $\phi_a$ defined by (\ref{eq:phia}), $a\in\Rnp$.

\begin{lemma} \label{lem:lnphia} $L_k(\phi_a)=(\max_{|J|=k}a_J)^{-1}$, where $a_J=a_{j_1}\cdot\ldots\cdot a_{j_k}$ for $J=(j_1,\ldots,j_k)$.
\end{lemma}

\begin{proof} For $k=n$ this is a well-known fact (see, for example, \cite{D1}). When $k<n$, we can use Siu's theorem stating that the Lelong number $\nu(T,0)$ of any positive closed current $T$ of bidegree $k$ equals the minimum of the Lelong numbers $\nu(T|_S,0)$ of its restrictions $T|_S$ to $k$-dimensional linear subspaces $S$, and the minimum attains at almost all subspaces $S$. For the current $T= (dd^c\phi_a)^k$ and almost all $S$, one has $u|_S=\max_{j\in I}{a_j}^{-1}\log|z_j|+O(1)$, where a $k$-tuple $I\subset\{1,\ldots,n\}$  is formed by the indices of the first $k$ largest components $a_j$ of the vector $a$.
Since for the restrictions one can use the aforementioned computation of the highest Lelong number, the lemma is proved.
\end{proof}

\begin{theorem} Let $u\in\MC$ be such that
$(dd^cu)^k$ is well defined near $0$ for some $k\le n$.
Then \beq\label{eq:mchln}\lambda_u\le k^{-1}L_k(u)^{1/k}.\eeq
\end{theorem}

\begin{proof}
For any $a\in\Rnp$, the bound $u\le\nu(u,a)\phi_a+O(1)$ implies, by
Demailly's Comparison Theorem,
\begin{equation}\label{eq:1bd} L_k(u)\ge [\nu(u,a)]^k
L_k(\phi_a).\end{equation} By Lemma \ref{lem:lnphia} and the inequality of arithmetic and geometric means,
\beq\label{eq:agmean} L_k(u)^{1/k} \ge
\frac{\nu(u,a)}{\max_{|J|=k}a_J^{1/k}}\ge
k\frac{\nu(u,a)}{\sum_i a_i}.\eeq
By Theorem~\ref{prop:lambdau}, this implies the desired bound.
\end{proof}

\medskip

{\it Remark.} As is easy to see, none of the relations (\ref{eq:mchln}), $k=1,\ldots,n$, can be deduced from the others, even for $u=\phi_a$. Let $a=(\delta,1,N)$ with $\delta<1<N$, then, denoting the right hand side of (\ref{eq:mchln}) by $\tau_k$, we get $\tau_1= \delta^{-1}$, $\tau_2=(4N)^{-1/2}$, and $\tau_3=(27\delta N)^{-1/3}$, so by varying the values $\delta$ and $N$ one can achieve all the orderings between $\tau_1$, $\tau_2$, and $\tau_3$.

Observe that, according to (\ref{eq:ltype}), $\lambda_{\phi_a}=(\sum_k a_k)^{-1}$.

\medskip

We can also characterize those multi-circled singularities $u$ for which inequalities (\ref{eq:mchln}) become equalities.

\begin{theorem}\label{theo:mceq} For a function $u\in\MC$, the equality
\beq\label{eq:mchln1}\lambda_u= k^{-1}L_k(u)^{1/k}\eeq
occurs if and only if its indicator $\Psi_u$ has the form \beq\label{eq:Bind} \Psi_u(z)=B\max_{j\in \bar J}\log|z_{j}|\eeq
for some $B\ge 0$ and a $k$-tuple $\bar J=(j_1,\ldots,j_k)\subset\{1,\ldots,n\}$.
\end{theorem}

\begin{proof} By Theorems \ref{eq: MAm} and \ref{prop:lambdau}, it suffices to prove the result for $u=\Psi_u\not\equiv 0$. Direct computations show that (\ref{eq:mchln1}) holds for the indicator (\ref{eq:Bind}), so let us prove the reverse statement.

We treat first the case $k=n$. By Theorem~\ref{prop:lambdau} and inequalities (\ref{eq:agmean}), one has
$$ L_n(u)^{1/n} =\sup_{a\in\Rnp} \frac{\nu(u,a)}{(a_1\cdot\ldots\cdot a_n)^{1/n}}=
n\sup_{a\in\Rnp}\frac{\nu(u,a)}{\sum_i a_i}=n\lambda_u.$$
Pick a sequence $a^{(s)}\in\{\Rnp:\: \sum_j a_j=1\}$ such that $\nu(u, a^{(s)})$ tends to $\lambda_u$. Let $\bar a$ be a limit point for $\{a^{(s)}\}_s$. By the continuity of the indicator $\Psi_u$, $\lambda_u=\nu(u,\bar a)$. By the inequality of arithmetic and geometric means,
$$(\bar a_1\cdot\ldots\cdot \bar a_n)^{1/n}\le n^{-1}\sum_i \bar a_i$$
with an equality if and only if all the $\bar a_i$ are equal. Therefore, $L_n(u)=[\nu(u)]^n$.

Since $\Psi_u\le\nu(u)\max_i\log|z_j|$ and their Monge-Amp\`ere masses at $0$ coincide, we get $\Psi_u=\nu(u)\max_i\log|z_j|$, which follows, for example, from \cite[Lemma 6.3]{R7}; one can however deduce it also from Corollary~\ref{cor:covol}.

Now let $k<n$. Using the same reasoning as above, we get $\lambda_u=\nu(u,\bar a)=\lim_{a\to\bar a}\nu(u,a)$ for some $\bar a\in\overline\Rnp$ with $\sum_j \bar a_j=1$. Since for any $k$-tuple $J$, in the relations
$$(\bar a_J)^{1/k}\le k^{-1}\sum_{j\in J}\bar a_j\le k^{-1}\sum_{1\le j\le n}\bar a_j$$
the second inequality is an equality if and only if $a_l=0$ for all $l\not\in J$, and the first one is an equality if and only if all the $a_j$ for $j\in J$ are equal, we deduce that, after a renumeration, $\bar a_j=k^{-1}$ for $1\le j\le k$ and $a_j=0$ for $k+1\le k\le n$.

So we have $\Psi_u\le\tilde\Psi_u=\nu(u,\bar a)\max_{1\le j\le k}\log|z_k|$ and, at the same time, $L_k(\Psi_u)=L_k(\tilde\Psi_u)$. As in the proof of Theorem~\ref{thm:MAn}, we use the relations $L_k(u|_S)=L_k(u)$ and $\Psi_{u|_S}=\Psi_u|_S$ for restrictions to almost all $k$-dimensional linear spaces $S$ such that $u|_S$ has isolated singularity and refer to the already proved case $k=n$.
\end{proof}

\section{Integrability index and higher Lelong numbers: general case}

To get inequalities (\ref{eq:1bd}) for arbitrary $u\in\PSH_0$, one can apply the known result on $L_n(u)$ \cite{D8} to restrictions of $u$ to "good" linear subspaces, a relation between the integrability indices of the function and its restriction being provided by a result from \cite{DK}.

\medskip

We will assume here the unbounded locus $UL(u)$ to be contained in an analytic set of codimension at least $k$, so the current $(dd^cu)^k$ will be well defined \cite{D1}.

\begin{theorem}\label{theo:gencasek}
If $UL(u)$ for a function $u\in\PSH_0$ is contained in an analytic set of codimension at least $k$, then
\beq\label{eq:mceq}\lambda_u\le k^{-1}L_k(u)^{1/k}\eeq
\end{theorem}

\begin{proof}
Let $UL(u)\subset Z$ for an analytic set $Z$ of codimension $l$, and take any $k\le l$. By Thie's theorem,
(see also \cite{D1}, Chapter 3, Thm.~5.8), there exist local coordinates
$x=(x',x'')$, $x'=(x_1,\ldots,x_k)$, $x''=(x_{k+1},\ldots,x_n)$,
and domains $U'\subset {\C}^j$, $U''\subset
{\C}^{n-j}$ such that $U'\times U''\Subset U$, $Z\cap(U'\times
U'')$ is contained in the cone $\{|x'|\le \gamma |x''|\}$ with
some constant $\gamma>0$, and the intersection $Z(x_0'')$ of the set
$U'\times\{x_0''\}$ with the variety $Z$ is at most finite for each $x_0''\subset U''$. (For $k=l$, the projection of $Z\cap(U'\times U'')$ onto $U''$ is a ramified covering with a finite number of
sheets.)

Since $Z(0'')=0''$, the function $u'(x')=u(x',0'')$ is locally bounded on $U'\setminus 0'$, so the main result of \cite{D8} -- namely, inequality (\ref{dem}) -- applied to $u'$ gives us the relation $$\lambda(u')\le k^{-1}\,{L_k(u')^{1/k}}.$$

We have $L_k(u')=(dd^cu')^k(0)=(dd^cu)^k\wedge (dd^c\log|z''|)^{n-k}$ which, by the choice of coordinates, equals $L_k(u)$.

Finally, by \cite[Prop.~2.2]{DK} (stating that the integrability index of a plurisubharmonic function does not exceed that of its restriction to any complex manifold), $\lambda(u)\le\lambda(u')$, and the proof is complete.
\end{proof}

\medskip

As a corollary, we get a bound on log canonical threshold for ideals in terms of their mixed Rees multiplicities in the sense of \cite{BA}.

 Let $\mi$ be the maximal ideal of the ring ${\mathcal O}_0$ of analytic germs at $0$, and $e(\cI_1,\ldots,\cI_n)$ be the mixed Samuel multiplicity of $\mi$-primary ideals $\cI_1,\ldots,\cI_n$ \cite{Te}. In \cite{BA} it was extended to any $n$-tuples of ideals $\cI_j$ as
 $$\sigma(\cI_1,\ldots,\cI_n)=\max_{r\in\Z_+}e(\cI_1+\mi^r,\ldots,\cI_n+\mi^r)$$
 and was shown to be finite under additional conditions on the ideals $\cI_j$. In particular, this is the case if $\cI_1=\ldots=\cI_k$ and $\cI_{k+1}=\ldots=\cI_n=\mi$, where
 $\cI$ is an ideal whose variety $V(\cI)$ has codimension $l\ge k$. We denote this mixed multiplicity by $e_k(\cI)$.

\begin{corollary}\label{cor:lct} If ${\rm codim}\, V(\cI)\ge k$, then
$$ lc({\cI})\ge\frac{k}{e_k({\cI})^{1/k}}.$$
\end{corollary}

\begin{proof}
If $f_1,\ldots, f_m$ are generators of the ideal $\cI$, then $e_k(\cI)$ equals $L_k(u)$ for the function $u=\max_j\log|f_j|$. For $k=n$ this was proved in \cite{D8}, and for $k<n$ it follows from the polarization formula (\ref{eq:polar}) for both the mixed Rees multiplicities and mixed Monge-Amp\`ere operators. Therefore the conclusion follows directly from Theorem~\ref{theo:gencasek}.
\end{proof}

\medskip
Unlike the case of multi-circled singularities (see Theorem~\ref{theo:mceq}), we do not know if an equality in (\ref{eq:mceq}) implies that the greenification $g_u$ satisfies $g_u(z)=B\max_{1\le s\le k}\log|z_{j_s}|$. As follows from what is proved in \cite{dFEM}, this is so for $u$ with isolated analytic singularity, i.e., $u=\log|F|+O(1)$ for a holomorphic mapping $F$ with isolated zero at $0$.

On the other hand, an equality in the {\sl lower} bound in Skoda's inequalities (\ref{skoda}) can be easily described in terms of the indicator $\Psi_u$. Observe that both the functions $u=\log|z|$ and $u=\log|z_1\cdot\ldots\cdot z_n|$ satisfy $\lambda_u=n^{-1}\nu_u$, which suggests that this class is quite large.

\begin{theorem} If $u\in\PSH_0$, then $\lambda_u=n^{-1}\nu_u$ if and only if the point $n^{-1}\nu_u{\bf 1}$  belongs  to the boundary of the indicator diagram $\Gamma_u$ (\ref{eq:charf1}) of $u$.
\end{theorem}

\begin{proof} Since $u$ is dominated near $0$, up to a bounded term, by its indicator $\Psi_u$, we have the chain of relations
$$ \lambda_{\Psi_u}\ge n^{-1}\nu_{\Psi_u}=n^{-1}\nu_u=\lambda_u\ge \lambda_{\Psi_u},$$ which
implies $\lambda_u=n^{-1}\nu_u$ if and only if $\lambda_{\Psi_u}= n^{-1}\nu_{\Psi_u}$, so we can assume $u=\Psi_u$. Then the representation (\ref{eq:lambdau}) shows that $\nu_u(a)\le n^{-1}\nu_u\sum a_j$ for any $a\in\Rnp$. In terms of the convex image $\widehat\Psi_u$ of the indicator, this rewrites as
\beq\label{eq: c1} \widehat\Psi_u(t)\ge \langle t, c{\bf 1}\rangle\quad \forall t\in\Rnm,\eeq
where $c= n^{-1}\nu_u$. In view of the representations (\ref{eq:charf}) and (\ref{eq:charf1}), this means precisely $c{\bf 1}\in\Gamma_u$. Moreover, for any $\epsilon\in (0,c)$, the point $(c-\epsilon){\bf 1}\not\in \Gamma_u$, because otherwise (\ref{eq: c1}) for $t=-{\bf 1}$ would give $-\nu_u\ge -\nu_u+n\epsilon$.
\end{proof}

\bigskip
In the end, we would like to notice that a slightly stronger variant of Demailly's inequality (\ref{dem}) can be deduced from (\ref{dFEM}) by means of a representation of the integrability index $\lambda_u$ obtained in \cite{FaJ2} (for $n=2$) and \cite{BFaJ} as
\beq\label{eq:bfj} \lambda_u=\sup_{V\in{\cal V}_{qm}}\frac{V(u)}{A(V)},\eeq
where ${\cal V}_{qm}$ is the collection of {\it quasi-monomial valuations} on $\PSH_0$ and $A(V)$ is {\it thinness} of a valuation $V$. For precise definitions we refer the reader to \cite{BFaJ}. An important fact to be used here is that any such valuation can be represented as a relative type (in the sense of \cite{R7}),
\beq\label{eq:reltype} V(u)= \liminf_{z\to 0}\frac{u(z)}{\phi(z)}\eeq
with a certain psh weight $\phi$, maximal outside $0$ and possessing a few additional properties. For example, the only multi-circular weights corresponding to quasi-monomial valuations are of the form $c\,\phi_a$ ($c>0$, $a\in\Rnp$), see \cite{R7}. In particular, such a representing weight $\phi$ is {\it tame}, which means that there exists a sequence of functions $\phi_j=\log|F_j|$ approximating $\phi$:
\beq\label{eq:demappr} \phi\le j^{-1}\phi_j+O(1)\le (1-C/j)\phi.\eeq
When $V$ is generated by $\phi_a$, we have actually $A(V)=\sum a_k$, so the representation (\ref{eq:lambdau}) is a particular case of (\ref{eq:bfj}).

Denote the collection of such {\it quasi-monomial weights} by $W_{qm}$.

\begin{theorem} For any $u\in\PSH_0$ such that $(dd^cu)^n$ is well defined near $0$,
\beq\label{eq:valbound} \lambda_u=\sup_\phi\lambda_\phi\le n^{-1}\sup_\phi L_n(\phi)^{1/n}\le n^{-1}L_n(u)^{1/n},\eeq
where the supremum is taken over all weights $\phi\in W_{qm}$ dominating $u$, that is,  $\phi\ge u+O(1)$.
\end{theorem}

\begin{proof}
Let $V\in {\cal V}_{qm}$ and $\phi\in W_{qm}$ be its representing weight. Since $u\le V(u)\phi+O(1)$ (which follows from (\ref{eq:reltype}) and the maximality of $\phi$), Demailly's Comparison theorem \cite{D1} gives
\beq\label{eq:b1} L_n(u)^{1/n}\ge V(u)L_n(\phi)^{1/n}.\eeq
In view of (\ref{dFEM}), $ \lambda_{\phi_j}\le n^{-1}L_n(\phi_j)^{1/n}$
and, using this together with (\ref{eq:demappr}), we get
$ L_n(\phi)^{1/n}\ge n(1-C/j)\lambda_{\phi}$
for each $j$ and so,
$$ L_n(\phi)^{1/n}\ge n\lambda_{\phi}.$$
By (\ref{eq:bfj}),
$$ \lambda_\phi\ge\frac{V(\phi)}{A(V)}=\frac1{A(V)},$$
so the estimation in (\ref{eq:b1}) continues to
$$L_n(u)^{1/n}\ge  nV(u)\lambda_\phi\ge n\frac{V(u)}{A(V)}.$$
Denote $\phi'=V(u)\phi\in W_{qm}$, then $u\le\phi'+O(1)$ and
$$n^{-1}L_n(u)^{1/n}\ge n^{-1}L_n(\phi')^{1/n}\ge \lambda_{\phi'}\ge\frac{V(u)}{A(V)}.$$
Finally, we use again (\ref{eq:bfj}) and the obvious relation $\lambda_{\phi'}\le \lambda_u$.
\end{proof}

\medskip

{\it Remarks.} 1. Note that (\ref{eq:valbound}) is a stronger relation than (\ref{dem}). For example, any quasi-monomial weight $\phi$ dominating the function $u=\max\{\log|z_1|, \frac14 \log|z_1z_2|,\log|z_2|\}$ is $\phi_a$ with $a=(a_1,a_2)\in A=\{a_1,a_2\ge 1,\ a_1+a_2\le 4\}$. Therefore,
$$\lambda_u=\sup_{a\in A}\lambda_{\phi_a}=\frac12\sup_{a\in A} \sqrt{L_2(\phi_a)}=\frac14<\frac1{2\sqrt2}= \frac12\sqrt{L_2(u)}.$$

2. Using the relation $L_k(u)^{1/k}\ge V(u)L_k(\phi)^{1/k}$ instead of (\ref{eq:b1}), we get a similar refinement of Theorem~\ref{theo:gencasek}:
$$ \lambda_u\le k^{-1}\sup_\phi L_k(\phi)^{1/k}\le k^{-1}L_k(u)^{1/k}.$$

\bigskip
{\small {\bf Acknowledgement.} The author thanks the anonymous referee for valuable suggestions.}

\vskip1cm

Tek/Nat, University of Stavanger, 4036 Stavanger, Norway

\vskip0.1cm

{\sc E-mail}: alexander.rashkovskii@uis.no

\end{document}